%% file: Robin_FW.tex
\documentclass{article}
\usepackage{spconf}

\usepackage{comment}

\usepackage{algorithm}
\usepackage{graphicx}
\usepackage[noend]{algpseudocode}
\usepackage{url}
\usepackage{breqn}
\usepackage{cite}
\usepackage[usenames]{color}
\usepackage{amsfonts}
\usepackage{fancyhdr}
\usepackage{todonotes}
\usepackage{arydshln}
\usepackage{amsmath,amssymb,amsthm}
\usepackage{psfrag,caption,subcaption}
\newtheorem{theorem}{Theorem}

\newtheorem{lemma}{Lemma}

\usepackage{arydshln}

\usepackage{hyperref}

\include{macros}

\begin{document}

\title{Faster Rates for the Frank-Wolfe Algorithm Using  Jacobi Polynomials}

\ninept
%
\name{Robin Francis and  Sundeep Prabhakar Chepuri 
\thanks{This work is supported in part by the CISCO CNI and Pratiskha Trust Fellowships.}
}
\address{Indian Institute of Science, Bangalore, India
}

\maketitle

\begin{abstract}
The Frank Wolfe algorithm (FW) is a popular projection-free alternative for solving large-scale constrained optimization problems. However, the FW algorithm suffers from a sublinear convergence rate when minimizing a smooth convex function over a compact convex set. Thus, exploring techniques that yield a faster convergence rate becomes crucial. A classic approach to obtain faster rates is to combine previous iterates to obtain the next iterate. In this work, we extend this approach to the FW setting and show that the optimal way to combine the past iterates is using a set of orthogonal Jacobi polynomials. We also a polynomial-based acceleration technique, referred to as Jacobi polynomial accelerated FW, which combines the current iterate with the past iterate using combing weights related to the Jacobi recursion. By carefully choosing parameters of the Jacobi polynomials, we obtain a faster sublinear convergence rate. We provide numerical experiments on real datasets to demonstrate the efficacy of the proposed algorithm. 

\end{abstract}


\maketitle

\section{Introduction} \label{sec:intro} 
Consider constrained optimization problems of the form 
\begin{equation}
\label{original problem}
    \underset{\vx}{\text{minimize}}\hspace{3mm}f(\vx) \quad \quad \text{subject to} \quad \quad \vx\in \mathcal{C},
\end{equation}
where $f: \normalfont{\textbf{dom}}(f) \rightarrow\mathbb{R}$ is a smooth convex function and the constraint set $\mathcal{C} \subseteq \normalfont{\textbf{dom}}(f)$ is compact and compact. Projected gradient descent~\cite{nestrov} is a popular technique to solve constrained optimization problems, wherein gradient descent iterates are projected onto the constraint set. The constraint set in many large-scale problems that are commonly encountered in machine learning and signal processing, such as lasso, matrix completion, regularized logistic or robust regression, has a nice structure, but the projection operation onto the constraint set can be computationally expensive. 

The Frank-Wolfe (FW) algorithm (also popularly known as the conditional gradient method)~\cite{FW1956} is a projection-free alternative to the projected gradient descent algorithm for solving constrained optimization problems. The FW algorithm avoids the projection step and instead has a direction-finding step that minimizes a local linear approximation of the objective function over the constraint set. 
The update is then computed by taking a convex combination of the previous iterate with the solution from the linear minimization step. The linear minimization step in FW is often easy to compute for $\ell_p$-norm ball or polytope constraint sets that frequently appear in machine learning and signal processing problems, thereby making FW an attractive choice.

In~\cite{revisitFW}, convergences results are provided for the FW algorithm. The FW algorithm has iteration complexity bounds that are independent of the problem size but limited by a sublinear convergence rate of $\mathcal{O}(1/k)$. With a carefully selected step size, for minimizing a strongly convex function over a strongly convex constraint set, the FW algorithm can achieve convergence rates of $\mathcal{O}(1/k^2)$ and linear convergence rate when the solution lies in the interior of the convex constraint set~\cite{garber2015faster}. The FW algorithm can even achieve linear rates under weaker conditions than strongly convex functions over polytopes by removing carefully chosen directions (referred to as away steps) that hinder the convergence rate~\cite{Awaystep}. By finding an optimal step size at each iteration, faster convergence of the FW algorithm may be obtained~\cite{garber2015faster,partan}.

For unconstrained optimization problems, Nesterov's acceleration technique~\cite{Nestrovk} and Polyak's heavy-ball momentum technqiue~\cite{POLYAK19641} provide faster rates of $\cO(1/k^2)$. Central to Polyak's heavy-ball  method~\cite{POLYAK19641} is the so-called \emph{multistep method}, which keeps track of all the previous iterates and computes successive updates by taking a weighted combination of the past iterates. Doing so yields a smoother trajectory and helps attain faster convergence rates. In a similar vein, several variants of the multistep method have been proposed for FW. In~\cite{AFW}, a momentum-guided accelerated FW (AFW) algorithm that can achieve a convergence rate of $\mathcal{O}(1/{k^2})$ for smooth convex minimization problems over $\ell_p$-norm balls, but an order of $\mathcal{O}(1/k)$ in the general setting has been proposed. The FW iterates zig-zag as we approach the optimal solution and might affect the convergence rate. In~\cite{partan}, the zig-zag trajectory is alleviated by using the concept of parallel tangents for FW, wherein a significant improvement in convergence has been reported but without any theoretical guarantees. In~\cite{weightedgrad_giannakis_icassap21}, the zig-zag phenomenon is controlled by keeping track of the past gradients and taking its weighted average with different heuristic choices for the weights.
When using multistep-based methods, finding the optimal updates and combining weights that can achieve faster convergence without much computational overhead is important.

The main aim of this work is, therefore, to determine an optimal technique independent of the problem size, to combine the previous iterates as in the multistep method, and obtain faster convergence rates. To this end, we draw inspiration from polynomial-based iteration methods for unconstrained problems~\cite{daspremont2021acceleration} and extend that to the FW setting. The main contributions of the paper are as follows.
\begin{itemize}
    \item We express the linear combination of the past FW iterates as a polynomial and show that the optimal polynomial that yields the maximum error reduction is given by a set of orthogonal Jacobi polynomials. 
    \item We then modify the FW algorithm to incorporate a Jacobi recursion, wherein the weights from the recurrence relation of the Jacobi polynomials are used to combine the past iterate with the current iterate.
    \item We show that for carefully selected parameters of the Jacobi polynomials, we obtain a faster convergence rate of $\cO(1/k^2)$ when minimizing a smooth convex function over a compact convex constraint set.
\end{itemize}

Through numerical experiments on real datasets, we demonstrate the efficacy of the proposed algorithm.

\section{Preliminaries} \label{sec:preliminaries}

We use boldface lowercase (uppercase) letters to represent vectors (respectively, matrices). We use $\langle \vx, \vy\rangle$ to denote the inner product between vectors $\vx$ and $\vy$. We denote an arbitrary norm using $\|\cdot\|$, $\ell_2$-norm using $||\cdot||_2$, and nuclear norm using $\|\cdot\|_\star$. We say that a function $f: \normalfont{\textbf{dom}}(f) \rightarrow \mathbb{R}$ is $L$ smooth over a convex set $\cC \subseteq \normalfont{\textbf{dom}}(f)$ if for all $\vx,\vy \in \cC$ it holds that
\[ \label{Lsmooth}
f(\vy) \leq  f(\vx) + \langle \nabla f(\vx), \vy-\vx\rangle + \frac{L}{2}\|\vy-\vx\|^2.
\]

\subsection{The FW method}

The \texttt{FW} algorithm~\cite{FW1956} finds an update direction based on the local linear approximation of $f$ around the current iterate $\vx_k$ by solving the linear minimization problem:
\begin{equation}
    \vs_{k} = \underset{\vs\in \mathcal{C}}{\mathrm{arg \,min}}  \hspace{2mm}f(\vx_k) + \left <\nabla f(\vx_k),\vs-\vx_k\right>.
\end{equation}
Then the update is computed by forming a convex combination of $\vs_{k}$ and $\vx_k$ as
\[
\vx_{k+1} = (1-\gamma_k)\vx_k + \gamma_k \vs_k
\]
with $\gamma_k \in [0,1]$. This ensures $\vx_{k+1} \in \cC$. Typically, a diminishing step size $\gamma_k$ is used~\cite{revisitFW}, but line search techniques to compute the optimal step size may also be used~\cite{FW1956}. The \texttt{FW} method is summarized as Algorithm~\ref{alg:FW}.

\begin{algorithm}[t]
\caption{The vanilla Frank-Wolfe method (\texttt{FW})}\label{alg:FW}
\begin{algorithmic}[1]
\State Initialize $\vx_0 \in \mathcal{C}$
\For{$k=0,1,\ldots$}
\State $\vs_{k} \gets \underset{\vs \in \mathcal{C}}{\mathrm{arg \,min}} \left <\nabla f(\vx_k),\vs\right> $ 
\State  $\vx_{k+1} \gets \vx_k + \gamma_{k}(\vs_k-\vx_k)$ 
\State $\gamma_k \gets \frac{2}{k+2}$
\EndFor
\end{algorithmic}
\end{algorithm}
The \texttt{FW} algorithm with a diminishing step size of $\gamma_k = \frac{2}{k+2}$ yields a sublinear convergence of $\cO(1/k)$ for minimizing $L$-smooth convex functions over a compact convex constraint set. Specifically, the suboptimality gap at the $k$th iterate of \texttt{FW} satisfies~\cite{revisitFW}
\[
f(\vx_k) - f(\vx^\star) \leq \frac{2LD^2}{k+2},
\]
where $D = \sup_{\vx,\vy \in \cC} \|\vx - \vy\|_2$ is the diameter of the constraint set and $\vx^\star$ is a minimizer of $f$.

\subsection{The Jacobi polynomials} \label{sec:jacobi_pre}
The Jacobi polynomials are a set of orthogonal polynomials, orthogonal with respect to (w.r.t) the weight $(1-x)^{\alpha}(1+x)^{\beta}$ in the interval $[-1,1]$, where $\alpha,\beta > -1$ are the tunable parameters. The Jacobi polynomials are defined by the second-order recurrence relation
\begin{align}
    \label{eq:Jacobi_recursion}
    J_{k+1}(\lambda) &= (a_{k}\lambda +  b_{k})J_{k}(\lambda) -  c_{k}J_{k-1}(\lambda), \quad k=1,2,\ldots
\end{align}
with $J_{0}(\lambda)=1$ and $J_{1}(\lambda) = a_{0}\lambda +b_{0}$. The recurrence weights are given by the formulas
\begin{align}
\label{eq:Jacobi_coefficients}
    a_{k} &= \frac{(\tau_k + k) (\tau_k + k + 1)}{2\tau_k(\tau_k- \beta)}\nonumber\\
    b_{k} &= \frac{(\tau_k + k)(\alpha^2-\beta^2)}{2\tau_k(\tau_k- \beta)(\tau_k + k - 1)}\nonumber\\
     c_{k} &= \frac{k(k+\beta)(\tau_k + k + 1) }{\tau_k(\tau_k- \beta)(\tau_k + k - 1)}
\end{align}
where $\tau_k = k + \alpha + \beta + 1$ and $a_{0} = \frac{\alpha+\beta+2}{2(1+\alpha)}$ and $b_{0} = \frac{\alpha-\beta}{2(1+\alpha)}.$
Furthermore, we have $a_k + b_k - c_k = 1.$

\section{The optimal polynomial} \label{sec:JFW}

A classic technique to accelerate iterative methods is the so-called \emph{multistep method}, which uses past iterates to construct the successive iterate~\cite{POLYAK19641}. In this section, we apply this technique to \texttt{FW} and find an optimal procedure to use the past iterates.

Suppose we take a linear combination of all the past iterates from Step $4$ of Algorithm \ref{alg:FW} with a fixed step size $\gamma$, we get
\begin{equation}
\label{eq:linear_comb}
    \vz_{k+1} = \sum_{i=0}^{k}\theta_i \vx_{i}= \sum_{i=0}^{k}\theta_i[(1- \gamma)\vx_{i-1} + \gamma\vs_{i-1}],
\end{equation}
where we constrain the combining coefficients such that $\sum_{i=0}^{k}\theta_i = 1$ with $\theta_i \in [0,1]$ for $i=0,1,\ldots,k$. This constraint ensures that the weighted average is a feasible point, i.e., $\vz_{k+1} \in \mathcal{C}$. 

Now an important question to ask is, can we find optimal coefficients that can possibly accelerate the FW algorithm? To answer this question, we express \eqref{eq:linear_comb} as a polynomial and then select the best polynomial that speeds up the convergence rate. To do so, let us
first unroll~\eqref{eq:linear_comb} and express it in terms of $\vx_0$ as
\begin{align*}
    \vz_{k+1} 
    &= \sum\limits_{i=0}^{k} \theta_i [(1-\gamma)^{i}\vx_0 + \gamma \sum\limits_{j=0}^{i-1} (1-\gamma)^{i-j}\vs_j].
    \end{align*}
Let us define the polynomial $p_k(1-\gamma)$ of degree $k$ as
\[
p_k(1-\gamma) = \theta_{k}(1-\gamma)^{k}+ \theta_{k-1}(1-\gamma)^{k-1} + \cdots + \theta_0,
\]
with $p_0(1-\gamma) = 1$. Then we can express $\vz_{k+1}$ as
\begin{align*}
\vz_{k+1} & = p_{k}(1-\gamma)\vx_0 + \gamma p_{k-1}(1-\gamma)\vs_0 + \cdots + \gamma p_{0}(1-\gamma)\vs_{k-1},
\end{align*}
where 
$
p_{k}(1-\gamma)+\gamma \sum_{i=1}^kp_{k-i}(1-\gamma) = 1 
$
and thus $\vz_{k+1} \in \textbf{conv}\{\vx_0,\vs_0,\cdots,\vs_{k-1}\}$. The best polynomial is the one that minimizes the distance of the $k+1$st iterate to $\vx^\star$, i.e., 
\begin{align*}
    \vz_{k+1} - \vx^\star &= p_{k}(1-\gamma)\vx_0 + \gamma p_{k-1}(1-\gamma)\vs_0  \\
    &\quad \quad \quad + \dots + \gamma p_{0}(1-\gamma) - \vx^\star \nonumber\\
      &= p_{k}(1-\gamma)(\vx_0 - \vx^\star) + \gamma p_{k-1}(\vs_0 - \vx^\star) \\
      & \quad \quad \quad +  \cdots + \gamma p_{0}(1-\gamma)(\vs_{k-1} - \vx^\star) 
\end{align*}
for each $k$, and is bounded from above as 
\begin{align}
    \|\vz_{k+1} - \vx^\star\| 
      &\leq \> |p_{k}(1-\gamma)|\,\|\vx_0 - \vx^\star\|  \nonumber\\
      &\quad \quad + |\gamma p_{k-1}(1-\gamma)|\,\|\vs_0 - \vx^\star\|
      \nonumber\\
      & \quad \quad  \,+\,  \cdots + |\gamma p_{0}(\gamma-1)|\,\|\vs_{k-1} - \vx^\star\|  \label{eq:poly_hk}
\end{align}
due to the triangle inequality. Thus minimizing this upper bound would provide updates that are as close as possible to the optimal path. 
Next, we find the best polynomial that minimizes the above bound in the following theorem.

\begin{theorem}
A $k$th degree orthogonal Jacobi polynomial is a minimizer of the upper bound in~\eqref{eq:poly_hk}.
\end{theorem}
\begin{proof}
Recall that the Jacobi polynomials are orthogonal in the interval $[-1,1]$. With the change of variable $\gamma = \frac{x + 1}{2}$ the Jacobi polynomials are orthogonal in the interval $[0,1]$ w.r.t. the weight $w(\gamma) = (2- 2\gamma)^\alpha (2\gamma)^\beta d\gamma$. The polynomial $p_k^\star$ that minimizes its norm w.r.t a weight function $w(\gamma)$, i.e.,
\begin{equation}
    p_k^\star \in \, \underset{p_k: p_0(1-\gamma) = 1}{\textrm{arg\,min}} \,\, \int_0^1 \,\, |p_k(1-\gamma)| \, (2- 2\gamma)^\alpha (2\gamma)^\beta d\gamma \label{eq:poly_mini}
\end{equation}
are given by the set of orthogonal polynomials~\cite{NEVAI19863}. 

Since a $k$th degree Jacobi polynomial can be expressed as a linear combination of lower degree Jacobi polynomials~\cite{jacobi_lincomb}, we further have $p_k(1-\gamma) = \sum_{i=0}^{k-1} c_i p_i(1-\gamma)$ for some combining weights $\{c_i\}$. Thus, for Jacobi polynomials, we have 
\begin{align*}
   &\underset{p_{\kappa}: p_0(1-\gamma) = 1, \kappa \leq k-1}{\text{minimize}} \,\, \int_0^1\, \sum_{i=0}^{k-1} |c_i p_i(1-\gamma)| \, (2- 2\gamma)^\alpha (2\gamma)^\beta d\gamma \nonumber\\  
    & \leq \,\underset{p_{\kappa}: p_0(1-\gamma) = 1, \kappa \leq k-1}{\text{minimize}} \sum_{i=0}^{k-1} \int_0^1\, |c_i|\,|p_i(1-\gamma)| (2- 2\gamma)^\alpha (2\gamma)^\beta d\gamma.\label{optima kth polynomial}
\end{align*}
Therefore, selecting a $k$th degree Jacobi polynomial ensures that the norm of the lower degree polynomials are also minimized. Thereby minimizing the upper bound in \eqref{eq:poly_mini} and the error $\|\vz_{k+1} - \vx^\star\|$ at the $k+1$st iteration.
\end{proof}
The above theorem suggests that combining {all the past iterates} using orthogonal Jacobi polynomials yields the best successive update. 
Based on this insight, we modify the vanilla Frank-Wolfe method in the next section. 

\section{Jacobi acceleration for Frank-Wolfe}

In this section, we present the proposed Jacobi accelerated Frank-Wolfe method (\texttt{JFW}), which is a polynomial-based technique for using the past iterates. For a convex and $L$-smooth function $f$, we show that to reach error of at most $\varepsilon$, for carefully chosen Jacobi polynomial parameters $(\alpha,\beta)$, \texttt{JFW} only needs $\mathcal{O}(1/\sqrt{\varepsilon})$ instead of $\mathcal{O}(1/{\varepsilon})$ as needed by \texttt{FW}.

\subsection{The {\normalfont{\texttt{JFW}}} algorithm}

We have seen in Section~\ref{sec:jacobi_pre} that orthogonal Jacobi polynomials follow a second-order recursion in~\eqref{eq:Jacobi_recursion} with three sequences of coefficients $\{a_k,b_k,c_k\}$ computed as in~\eqref{eq:Jacobi_coefficients}. This means that, we do not have to store all the past iterates to compute~\eqref{eq:linear_comb}, but only store the previous iterate. Specifically, the proposed {\normalfont{\texttt{JFW}}} algorithm computes an intermediate update
\[
\vy_{k+1} = \vx_k + \gamma_k (\vs_k - \vx_k),
\]
where $\vs_k$ is the \texttt{FW} direction computed as in Step~3 of Algorithm~\ref{alg:FW} with $\gamma_k = 2/(k+2)$. Then we combine the intermediate update with the past iterate using the second-order Jacobi recursion [cf.~\eqref{eq:Jacobi_recursion}]:
\begin{equation}
   \vz_{k+1} = (a_k(1-\gamma) + b_k)\vx_{k} - c_k \vx_{k-1}
   \label{eq:}
\end{equation}
with $\gamma \in [0,1]$.
Since $a_k(1-\gamma) + b_k - c_k = 1 -\gamma a_k$ for the Jacobi polynomials, $\vz_{k+1} \notin \mathcal{C}$. Therefore, we make a final correction step
\[
\vx_{k+1} = \vz_{k+1} + \gamma a_k\vx_{k}
\]
so that $\vx_{k+1} \in \cC$ is a feasible point. 
The proposed \texttt{JFW} is summarized as Algorithm~\ref{alg:JFW}.

\begin{algorithm}[t]
\caption{Jacobi accelerated Frank-Wolfe (\texttt{JFW})}\label{alg:JFW}
\begin{algorithmic}[1]
\State Initialize $\vx_0 \in \mathcal{C},$ $\alpha\geq\beta>-1$, and $\gamma$  
\For{$k=0,1,\ldots$}
\State  $\vs_{k} \gets \underset{\vs\in \mathcal{C}}{\mathrm{arg \,min}}  \hspace{2mm} \left <\nabla f(\vx_k),\vs\right>$
\Comment {\footnotesize {FW direction finding}}
\State    $\vy_{k+1}  \gets \vx_k + \gamma_{k}(\vs_{k}-\vx_k)$
\Comment {\footnotesize {FW update}}

\State     $\vz_{k+1} \gets (a_k(1-\gamma) + b_k)\vy_{k+1} - c_k \vx_k$
\Comment {\footnotesize {Jacobi recursion}}

\State $\vx_{k+1} \gets \vz_{k+1} + \gamma a_k\vx_k$
\Comment {\footnotesize {Correction step}}
\State    $\gamma_k\gets \frac{2}{k+2}$
\EndFor
\end{algorithmic}
\end{algorithm}

\begin{figure*}[ht]
\centering
\subfloat[\label{fig:logistic}]{\includegraphics[width=0.66\columnwidth]{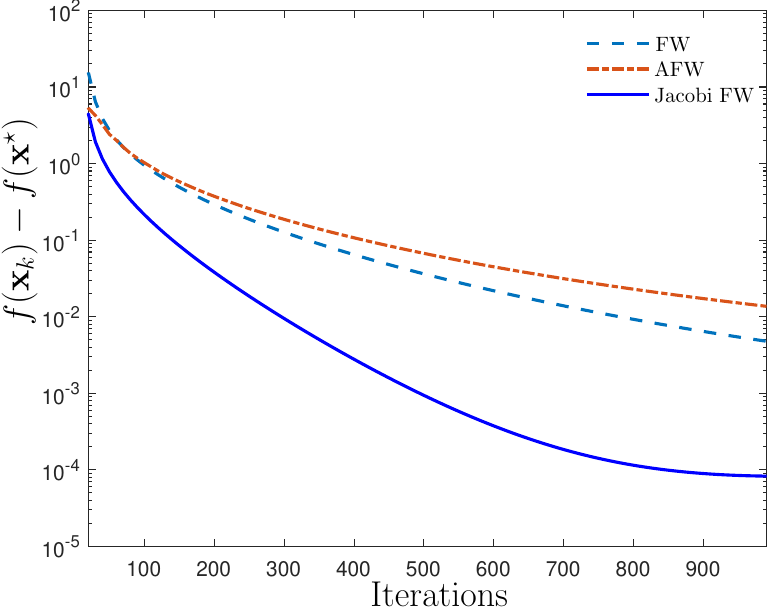}}
~
\subfloat[\label{fig:pima}]{\includegraphics[width=0.66\columnwidth]{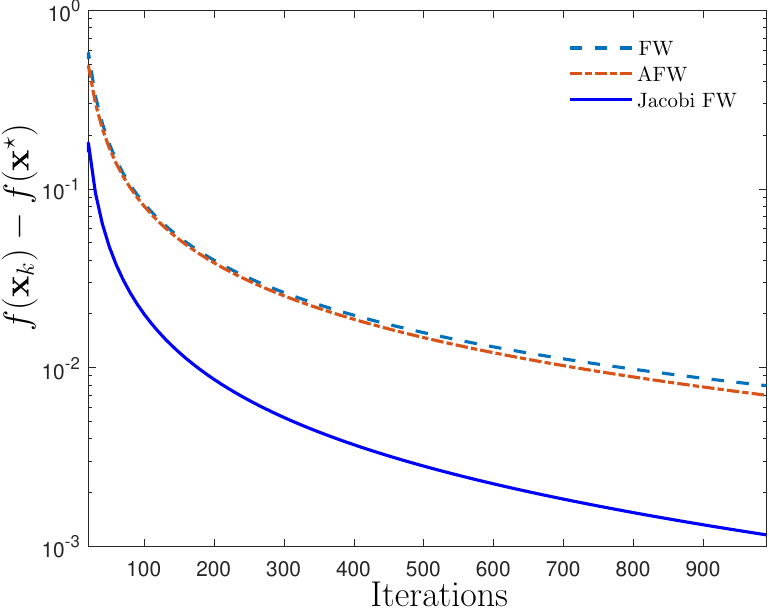}}
~
\subfloat[\label{fig:matrix_comp}]{\includegraphics[width=0.66\columnwidth]{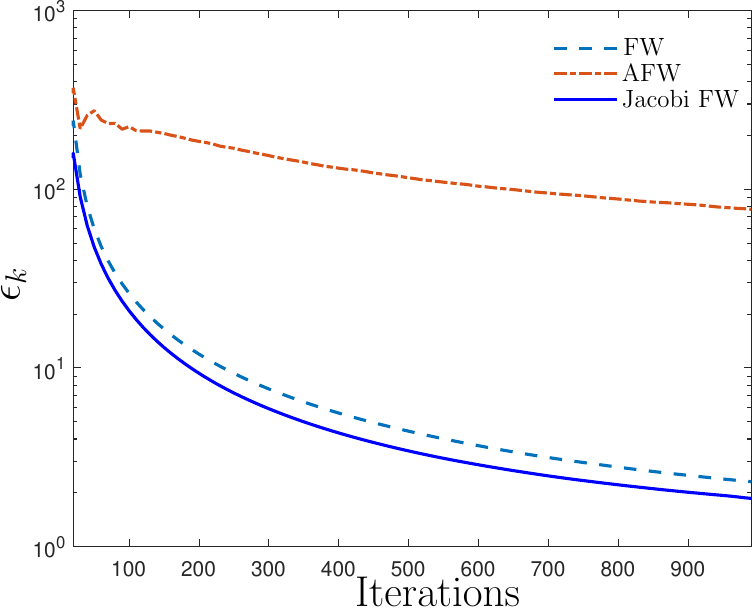}}
\caption{\small Performance of the \texttt{FW}, \texttt{AFW} and \texttt{JFW} algorithms on (a) {{Breast cancer} dataset for $\ell_2$-constrained logistic regression, (b) {Pima Indian diabetes} dataset for robust ridge regression, and (c) {Movielens 100k}} dataset for robust matrix completion task.}
\label{fig:plot1}
\end{figure*}

\subsection{Convergence analysis}
To demonstrate the convergence of the Jacobi FW, we consider a $L$-smooth convex function over a compact convex constraint set. The next theorem states the improvement in convergence of the proposed algorithm.

\begin{theorem}\label{thm:cvx} 
Let $f: \normalfont{\textbf{dom}}(f) \rightarrow \mathbb{R}$ be a $L$-smooth and convex function, $\cC \subseteq \normalfont{\textbf{dom}}(f)$ be compact and convex, and $\vx^\star$ be a minimizer of $f$ over $\cC$. For appropriately chosen parameters $(\alpha,\beta,\gamma)$ with  $\alpha \geq \beta > -1$, \texttt{JFW} in Algorithm~\ref{alg:JFW} satisfies
\begin{equation}
     f(\vx_k) - f(\vx^\star) \leq \left|\frac{\alpha}{\beta}\right|\frac{4LD^2}{(k+1)(k+2)},
\end{equation}
where $D = \sup_{\vx,\vy \in \cC} \|\vx - \vy\|_2$ is the diameter of the constraint set.
\end{theorem}

The complete proof is available online\footnote{ \href{https://drive.google.com/file/d/1BfP-iDiAZ0cKZ1OzJA5X1_6MGeDvq7wE/view}{\texttt{\url{https://drive.google.com/file/d/1BfP-iDiAZ0cKZ1OzJA5X1_6MGeDvq7wE/view}}}} and is omitted due to space constraints. Here, we give a brief sketch of the proof, which is based on mathematical induction. For the base case $k=0$, the bound directly follows from definitions of the diameter of the set and $L$-smoothness of the objective function. Then by the induction hypothesis, we assume that the bound holds for $k$ iterations and show that the bound is also valid for approximately chosen tuning parameters $(\alpha,\beta)$ at the $k+1$st iterate. To do so, we show that JFW satisfies the descent property
\[
f(\vx_{k+1}) \leq f(\vx_{k}) - \frac{6\omega_kLD^2}{(k+2)^2},
\]
with $\omega_k = a_k(1-\gamma) + b_k \in [0,1]$ and $L\geq0$.

Theorem~\ref{thm:cvx} affirms that we can achieve a faster sublinear convergence rate of $\cO(1/k^2)$ using a Jacobi polynomial-based acceleration technique as compared to the $\cO(1/k)$ rate of FW. However, in order to achieve this faster rate, careful tuning of the parameters of the Jacobi polynomials, $(\alpha,\beta, \gamma)$, is needed. Further, the bound on iteration complexity of \texttt{JFW} matches the bounds proposed in \cite{AFW}, but for more general convex and compact constraint sets.

\section{Numerical experiments}\label{sec:Numerical Exp}

We validate the \texttt{JFW} algorithm on real datasets  for three tasks, namely, logistic and robust regression with an $\ell_2$-norm constraint and robust matrix completion. We compare \texttt{JFW} with the vanilla \texttt{FW} in Algorithm~\ref{alg:FW} and the momentum-guided \texttt{AFW} \cite{AFW} algorithms\footnote{Software to reproduce the plots in the paper is available at~\href{https://colab.research.google.com/drive/1P1_DsP4nN2YYykWanNuasdbfQ4j6EAhB?usp=sharing}{\texttt{\url{https://colab.research.google.com/drive/1P1_DsP4nN2YYykWanNuasdbfQ4j6EAhB?usp=sharing}}}}.

\subsection{Logistic regression}
Consider the logistic regression problem with an $\ell_2$-norm constraint for binary classification:
\begin{align*} 
   &\underset{\vx}{\text{minimize}}\quad \frac{1}{m}\sum_{i=1}^{m}\log(1 + \exp{(-b_i\left<\mathbf{a}_i,\vx\right>)}) \\ 
   &\text{s. to \quad} \quad ||\vx||_2 \leq t,
\end{align*}
where $\mathbf{a}_i \in \mathbb{R}^d$ is the feature vector and $b_i$ is the binary class label. We apply \texttt{FW}, \texttt{AFW} and \texttt{JFW} on a {{breast cancer}}~\cite{breastcancerUCI} dataset with $m=698$ and $d=9$. We use $\alpha = \beta = 1.2$, $t=50$, and $\gamma = 2/3$. The suboptimality gap is shown in Fig.~\ref{fig:plot1}(a), where to compute the suboptimality gap we obtain the optimal value using {\tt{cvxpy}}~\cite{diamond2016cvxpy1}.
We can see from Fig.~\ref{fig:plot1}(a) that \texttt{JFW} has a faster rate.

\subsection{Robust regression}
Next, we consider robust ridge regression using a smooth Huber loss function with an $\ell_2$-norm constraint~\cite{owen2007robust}:
\begin{align*} 
   &\underset{\vx}{\text{minimize}}\quad \frac{1}{m}\sum_{i=1}^{m}H_{\delta}(y_i - \langle\va_i,\vx\rangle) \\ 
   &\text{s. to \quad} \quad ||\vx||_2 \leq t,
\end{align*}
where $\va_i \in \mathbb{R}^d$ and $y_i\in \mathbb{R}$. The Huber loss function is given by
\[
  H_{\delta}(c) =
  \begin{cases}
         c^2, & \text{if $|c|\leq \delta$}, \\
         2\delta|c| - \delta^2, & \text{otherwise.} 
  \end{cases}
\]

We apply robust ridge regression on {{Pima Indian diabetes}}~\cite{pimadataset} dataset with $m=767$ and $d= 8$, and the optimal function value is computed using {{scikit-learn}} \cite{scikit-learn}. We use $t=35$, $\delta=0.5$, $\alpha = \beta = 1450$, and $\gamma=0.65$. Fig.~\ref{fig:plot1}(b) shows the suboptimality gap, where we see that \texttt{JFW} has a faster convergence than \texttt{FW} and \texttt{AFW} as before.

\subsection{Robust matrix completion}
Finally, we consider the robust matrix completion problem
to predict unobserved missing entries from observations corrupted with outliers. We use the Huber loss function as before along with a low-rank promoting nuclear norm constraint for the robust matrix completion task, i.e., we solve
\begin{align*} 
   & \underset{\mX}{\text{minimize}}\,\,  \sum_{(i,j) \in \Omega} H_{\delta}(A_{ij} - X_{ij}) \\ 
   &\text{s. to \quad} \quad \|\mX\|_\star\leq R,
\end{align*}
where $\mathbf{A} \in \mathbb{R}^{m \times n}$ is the partially observed matrix with the entries ${A}_{ij}$ available for $(i,j) \in \Omega$. The sampling mask $\Omega$ is know. 

We apply robust matrix completion on the {{Movielens 100K}}~\cite{movielensUCI} dataset with $m = 1682$ movies rated by $n = 943$ users with $6.30\%$ percent ratings observed. Further, we introduce outliers by setting $4\%$ of the total ratings (i.e., $mn$ ratings) to its maximum value. We use 50\% of the available ratings for training and the remaining observations for testing. We use $R = 5$, $\alpha=\beta=4.5$, $\gamma = 2/3$ and $\delta = 4$. Fig.~\ref{fig:plot1}(c) shows normalized error, which is defined as 
\[
\epsilon_k = \frac{\sum_{(i,j) \in \Omega_{\rm test}} H_\delta(A_{ij} - X_{ij})}
{\sum_{(i,j) \in \Omega_{\rm test}} H_\delta(A_{ij})}
\]
with $\Omega_{\rm test}$ being the index set of the observations in the test set. We again observe that \texttt{JFW} performs slightly better than \texttt{FW}.

\section{Conclusions} 

The FW algorithm is a viable alternative for solving constrained optimization problems for which projection onto the constraint set is a computationally expensive operation. Motivated by polynomial-based acceleration techniques for unconstrained problems, we show that a set of orthogonal Jacobi polynomials with carefully chosen parameters can yield faster convergence rates. We have proposed Jacobi accelerated FW algorithm that provably has a faster sublinear convergence rate for minimizing a smooth convex function over a compact convex constraint set.

\pagebreak

\bibliographystyle{IEEEtran}
\bibliography{references}

\end{document}

%% file: macros.tex
\def\cast{{
   \mathord{
      \hbox to 0em{
         \ooalign{
	   \smash{\hbox{$\ast$}}\crcr
	   \smash{\hskip-1pt\Large\hbox{$\circ$}} }
	 \hidewidth}
      \phantom{\bigcirc}
} }}

\newcommand{\bds}{\begin {itemize}}
\newcommand{\eds}{\end {itemize}}
\newcommand{\bdf}{\begin{definition}}
\newcommand{\blm}{\begin{lemma}}
\newcommand{\edf}{\end{definition}}
\newcommand{\elm}{\end{lemma}}
\newcommand{\bthm}{\begin{theorem}}
\newcommand{\ethm}{\end{theorem}}
\newcommand{\bprp}{\begin{prop}}
\newcommand{\eprp}{\end{prop}}
\newcommand{\bcl}{\begin{claim}}
\newcommand{\ecl}{\end{claim}}
\newcommand{\bcr}{\begin{coro}}
\newcommand{\ecr}{\end{coro}}
\newcommand{\bquest}{\begin{question}}
\newcommand{\equest}{\end{question}}

\newcommand{\larrow}{{\larrow}}

\newcommand{\argmin}{\ensuremath{\mathrm{arg}\min}}
\newcommand{\argmax}{\ensuremath{\mathrm{arg}\max}}

\newcommand{\cC}{{\ensuremath{\mathcal{C}}}}

\newcommand{\cO}{{\ensuremath{\mathcal{O}}}}

\newcommand{\va}{{\ensuremath{{\mathbf{a}}}}}

\newcommand{\vs}{{\ensuremath{{\mathbf{s}}}}}

\newcommand{\vx}{{\ensuremath{{\mathbf{x}}}}}

\newcommand{\vy}{{\ensuremath{{\mathbf{y}}}}}
\newcommand{\vz}{{\ensuremath{{\mathbf{z}}}}}

\newcommand{\mX}{{\ensuremath{\mathbf{X}}}}

\usepackage{latexsym}

\def\IC{\mathbb C}
\def\IN{\mathbb N}
\def\IZ{\mathbb Z}
\def\IR{\mathbb R}

\def\shat{^{\mathchoice{}{}%
 {\,\,\smash{\hbox{\lower4pt\hbox{$\widehat{\null}$}}}}%
 {\,\smash{\hbox{\lower3pt\hbox{$\hat{\null}$}}}}}}

\def\bSigma{{
      \ooalign{
      \smash{\hskip.4pt\raise.4pt\hbox{$\Sigma$}}\vphantom{}\crcr
      \smash{\hskip.7pt\raise.6pt\hbox{$\Sigma$}}\vphantom{}\crcr
      \smash{\hbox{$\Sigma$}}\vphantom{$\Sigma$}}
      \vphantom{\hbox{$\Sigma$}}
      }}
\def\bTheta{{
      \ooalign{
      \smash{\hskip.5pt\raise.5pt\hbox{$\Theta$}}\vphantom{}\crcr
      \smash{\hskip.0pt\raise.1pt\hbox{$\Theta$}}\vphantom{}\crcr
      \smash{\hbox{$\Theta$}}\vphantom{$\Theta$}}
      \vphantom{\hbox{$\Theta$}}
      }}
\def\bDelta{{
      \ooalign{
      \smash{\hskip.4pt\raise.4pt\hbox{$\Delta$}}\vphantom{}\crcr
      \smash{\hskip.7pt\raise.6pt\hbox{$\Delta$}}\vphantom{}\crcr
      \smash{\hbox{$\Delta$}}\vphantom{$\Delta$}}
      \vphantom{\hbox{$\Delta$}}
      }}
\def\bLambda{{
      \ooalign{
      \smash{\hskip.5pt\raise.5pt\hbox{$\Lambda$}}\vphantom{}\crcr
      \smash{\hskip.0pt\raise.1pt\hbox{$\Lambda$}}\vphantom{}\crcr
      \smash{\hbox{$\Lambda$}}\vphantom{$\Lambda$}}
      \vphantom{\hbox{$\Lambda$}}
      }}

\makeatletter

\def\bordermatrix#1{\begingroup \m@th
  \@tempdima 8.75\p@
  \setbox\z@\vbox{%
    \def\cr{\crcr\noalign{\kern2\p@\global\let\cr\endline}}%
    \ialign{$##$\hfil\kern2\p@\kern\@tempdima&\thinspace\hfil$##$\hfil
      &&\quad\hfil$##$\hfil\crcr
      \omit\strut\hfil\crcr\noalign{\kern-\baselineskip}%
      #1\crcr\omit\strut\cr}}%
  \setbox\tw@\vbox{\unvcopy\z@\global\setbox\@ne\lastbox}%
  \setbox\tw@\hbox{\unhbox\@ne\unskip\global\setbox\@ne\lastbox}%
  \setbox\tw@\hbox{$\kern\wd\@ne\kern-\@tempdima\left[\kern-\wd\@ne
    \global\setbox\@ne\vbox{\box\@ne\kern2\p@}%
    \vcenter{\kern-\ht\@ne\unvbox\z@\kern-\baselineskip}\,\right]$}%
  \null\;\vbox{\kern\ht\@ne\box\tw@}\endgroup}
\makeatother

\makeatletter
\def\argmin{\mathop{\operator@font arg\,min}}
\def\argmax{\mathop{\operator@font arg\,max}}
\makeatother

\newcommand{\bea}{\begin{array}}
\newcommand{\ena}{\end{array}}
\newcommand{\beq}{\begin{equation}}
\newcommand{\enq}{\end{equation}}

\newcommand{\beqa}{\begin{eqnarray}}
\newcommand{\enqa}{\end{eqnarray}}

\newcommand{\beqan}{\begin{eqnarray*}}
\newcommand{\enqan}{\end{eqnarray*}}

\newcommand{\AL}{\begin{enumerate}}
\newcommand{\ALE}{\end{enumerate}}

\def\addots{\mathinner{
    \mkern1mu\raise0pt\vbox{\kern7pt\hbox{.}}
    \mkern2mu\raise4pt\hbox{.}
    \mkern2mu\raise7pt\hbox{.}
    \mkern1mu}}

\def\sddots{\mathinner{
    \mkern.8mu\raise7pt\hbox{.}
    \mkern.8mu\raise4pt\hbox{.}
    \mkern.8mu\raise0pt\vbox{\kern7pt\hbox{.}}
    \mkern1mu}}

\def\saddots{\mathinner{
    \mkern.2mu\raise0pt\vbox{\kern7pt\hbox{.}}
    \mkern.2mu\raise4pt\hbox{.}
    \mkern.2mu\raise7pt\hbox{.}
    \mkern1mu}}

\def\sqplus{\mathbin{
	{\ooalign{\hfil\raise.3ex\hbox{\scriptsize
	+}\hfil\crcr\mathhexbox274\crcr\mathhexbox275}}
	}} 
\def\sqminus{\mathbin{
	{\ooalign{\hfil\raise.3ex\hbox{\scriptsize
	--}\hfil\crcr\mathhexbox274\crcr\mathhexbox275}}
	}}

\def\IC{{
   \mathord{
      \hbox to 0em{
	 \hskip-4pt
         \ooalign{
	   \smash{\hskip1.9pt\raise2.6pt\hbox{$\scriptscriptstyle |$}}\crcr
	   \smash{\hbox{\rm\sf C}} }
	 \hidewidth}
      \phantom{\hbox{\rm\sf C}}
} }}
\def\IN{
    {\ooalign{
   \smash{\hskip2.2pt\raise1.5pt\hbox{$\scriptscriptstyle |$}}\vphantom{}\crcr
   \hbox{\sf N}
	}}
	} 
\def\IZ{
    {\ooalign{
   \smash{\hskip1.9pt\raise0pt\hbox{$\sf Z$}}\vphantom{}\crcr
   \hbox{\sf Z}
	}}
	} 
\def\IR{
    {\ooalign{
   \smash{\hskip2.2pt\raise1.5pt\hbox{$\scriptscriptstyle |$}}\vphantom{}\crcr
   \smash{\hskip2.2pt\raise3.3pt\hbox{$\scriptscriptstyle |$}}\vphantom{}\crcr
   \hbox{\sf R}
	}}
	} 

\DeclareMathAlphabet{\mathcmb}{OT1}{cmr}{b}{n}

\def\bSigma{\ensuremath{\mathcmb{\Sigma}}}
\def\bLambda{\ensuremath{\mathcmb{\Lambda}}}

\def\bTheta{\ensuremath{\mathcmb{\Theta}}}

\newcommand{\SI}{\begin{indlist}}
\newcommand{\EI}{\end{indlist}}

\newcommand{\DL}{\begin{dashlist}}
\newcommand{\DLE}{\end{dashlist}}

\makeatletter
\def\setboxz@h{\setbox\z@\hbox}
\def\wdz@{\wd\z@}
\def\boxz@{\box\z@}
\def\underset#1#2{\binrel@{#2}%
  \binrel@@{\mathop{\kern\z@#2}\limits_{#1}}}
\def\binrel@#1{\begingroup
  \setboxz@h{\thinmuskip0mu
    \medmuskip\m@ne mu\thickmuskip\@ne mu
    \setbox\tw@\hbox{$#1\m@th$}\kern-\wd\tw@
    ${}#1{}\m@th$}%
  \edef\@tempa{\endgroup\let\noexpand\binrel@@
    \ifdim\wdz@<\z@ \mathbin
    \else\ifdim\wdz@>\z@ \mathrel
    \else \relax\fi\fi}%
  \@tempa
}
\let\binrel@@\relax%
\makeatother
